\documentclass[12pt]{smfart}
\usepackage[french]{babel}
\usepackage{amscd,verbatim,amssymb}

\newcommand{\C}{\mathbf{C}}
\newcommand{\F}{\mathbf{F}}
\newcommand{\Q}{\mathbf{Q}}
\newcommand{\Z}{\mathbf{Z}}

\newcommand{\fD}{\mathfrak{D}}
\newcommand{\fa}{\mathfrak{a}}
\newcommand{\fb}{\mathfrak{b}}
\newcommand{\ff}{\mathfrak{f}}
\newcommand{\fp}{\mathfrak{p}}

\renewcommand{\phi}{\varphi}
\renewcommand{\epsilon}{\varepsilon}

\newcommand{\pmods}[1]{\; (\operatorname{mod}^* #1)}
\newcommand{\Tr}{\operatorname{Tr}}
\newcommand{\Ind}{\operatorname{Ind}}
\newcommand{\Hom}{\operatorname{Hom}}
\newcommand{\Ker}{\operatorname{Ker}}

\newcommand{\ab}{{\operatorname{ab}}}
\newcommand{\pgcd}{\operatorname{pgcd}}
\renewcommand{\det}{\operatorname{det}}

\newcounter{spec}
\newenvironment{thlist}{\begin{list}{\rm{(\roman{spec})}}%
{\usecounter{spec}\labelwidth=20pt\itemindent=0pt\labelsep=10pt}}%
{\end{list}}%

\newcommand{\Inj}{\lhook\joinrel\longrightarrow}

\newcommand{\by}[1]{\overset{#1}{\longrightarrow}}
\newcommand{\iso}{\by{\sim}}

\newtheorem{lemme}{Lemme}

\newtheorem{prop}{Proposition}
\newtheorem{cor}{Corollaire}
\theoremstyle{definition}
\newtheorem{defn}{D\'efinition}

\theoremstyle{remark}
\newtheorem{rque}{Remarque}

\begin{document}

\title{Quelques calculs de sommes de Gauss}
\author{Bruno Kahn}
\address{Institut de Math{\'e}matiques de Jussieu\\ UPMC - UFR 929, Ma\-th\'e\-ma\-ti\-ques\\ 4
Place Jussieu\\75005 Paris\\ France}
\email{kahn@math.jussieu.fr}
\date{28 juillet 2011}
\begin{abstract}
On remarque que l'action galoisienne sur les constantes locales associ\'ees aux
repr\'esentations galoisiennes d'un corps local fournit des renseignements sur leur nature
arithm\'etique, permettant notamment de borner leur ordre quand ce sont des racines de
l'unit\'e. Elle fournit aussi des renseignments sur l'effet des op\'erations d'Adams sur ces
constantes. 
\end{abstract}
\begin{altabstract}
We observe that the Galois action on local constants associated to Galois representations of a
local field yields information on their arithmetic nature, for example provides an upper
bound to their order when they are roots of unity. It also yields information on the effect of
Adams operations on these constants.
\end{altabstract}
\maketitle

\hfill \`A Paulo Ribenbo\"\i m

\section*{Introduction} Les constantes locales, ou facteurs epsilon locaux, associ\'es \`a une
repr\'esentation galoisienne complexe d'un corps local, ont \'et\'e relativement peu
\'etudi\'es pour eux-m\^emes. Rappelons qu'il s'agit d'une d\'ecomposition canonique de la constante de
l'\'equation fonctionnelle des fonctions $L$ d'Artin en produits de facteurs locaux,
g\'en\'eralisant la d\'ecomposition obtenue \`a partir de la th\`ese de
Tate \cite{tatethese} dans le cas ab\'elien.

Plus pr\'ecis\'ement, si $K$ est un corps global, la fonction $L$ d'Artin compl\'et\'ee
$\Lambda(\rho,s)$ associ\'ee \`a  une repr\'esentation complexe $\rho$ du groupe de Galois
absolu de $K$ admet (d'apr\`es Artin et Brauer) un prolongement m\'eromorphe au plan complexe et
une
\'equation fonctionnelle de la forme
\[\Lambda(\rho,1-s) = W(\rho)\Lambda(\rho^*,s)\]
o\`u $\rho^*$ est la repr\'esentation duale de $\rho$ et $W(\rho)$ est un nombre complexe de
module $1$ \cite[p. 14]{martinet}. De plus $W(\rho)\in \Q^\ab$, la cl\^oture cyclotomique de
$\Q$.
 
Lorsque $\rho$ est de degr\'e $1$, le corps de classes et la th\`ese de Tate donnent une
d\'ecomposition \cite[p. 346]{tatethese}
 \begin{equation}\label{eq1}
 W(\rho)=\prod_v W(\rho_v)
 \end{equation}
dans laquelle $v$ d\'ecrit l'ensemble des places de $K$ et $\rho_v$ est la restriction de
$\rho$ au groupe de d\'ecomposition en $v$. Le nombre complexe $W(\rho_v)$ ne d\'epend que de
$\rho_v$, est de module $1$ et vaut $1$ si $\rho$ est non ramifi\'e en $v$.

Langlands a d\'emontr\'e que la d\'ecomposition \eqref{eq1} s'\'etendait au cas non
ab\'elien.  Sa d\'emonstration, non publi\'ee, est r\'esum\'ee dans un manuscrit non publi\'e
de Helmut Koch \cite{koch}. Une autre d\'emonstration en a \'et\'e donn\'ee par Deligne
\cite{deligne}; une variante de celle-ci appara\^\i t au \S 2 de \cite{tate}, qui me servira de
r\'ef\'erence principale dans cette note.  Dans le cas de caract\'eristique positive, une
troisi\`eme preuve est donn\'ee par Laumon dans \cite{laumon}.

Rappelons que, si maintenant $K$ est un corps local, la fonction $\rho\mapsto W(\rho)$ est
caract\'eris\'ee par les axiomes suivants

\begin{thlist}
\item \emph{Additivit\'e}: $W(\rho+\rho') = W(\rho)W(\rho')$.
\item \emph{Inductivit\'e en degr\'e $0$}: si $L/K$ est une extension finie, et si $\rho$ est
un caract\`ere virtuel du groupe de Galois absolu $G_L$, de degr\'e $0$, alors
$W(\Ind_{G_L}^{G_K}\rho) = W(\rho)$.
\item \emph{Normalisation}: Si $\deg \rho=1$, $W(\rho)$ est donn\'e par la formule de \cite[p.
322]{tatethese} via le corps de classes local. 
\end{thlist}

Voici, \`a ma connaissance, ce qui appara\^\i t dans la litt\'erature sur les nombres $W(\rho)$:
 
\begin{enumerate}
\item Une \'etude d\'etaill\'ee de $W(\rho)$ est faite par Fr\"ohlich dans \cite{frohlich}
lorsque $\rho$ est mod\'er\'ement ramifi\'ee.
\item Si $\rho$ est irr\'eductible et sauvagement ramifi\'ee (c'est-\`a-dire que l'exposant de
son conducteur d'Artin est
$>1)$, $W(\rho)$ est une racine de l'unit\'e \cite[cor. 4]{tate}. Cela r\'esulte du cas
de degr\'e $1$, o\`u le r\'esultat est d\^u \`a Lamprecht et Dwork (ibid., cor. 1). Dans
\cite[1.4]{GK}, G\'erardin et Kutzko donnent de ce cas particulier une preuve plus explicite
que celle figurant dans \cite{tate}.
\item Si $\rho$ est orthogonale, virtuelle de degr\'e $0$ et de d\'eterminant $1$,
$W(\rho)^2=1$; dans \cite{deligne2}, Deligne donne la formule
\[W(\rho)=(-1)^{w_2(\rho)}\]
o\`u $w_2(\rho)$ est la \emph{deuxi\`eme classe de Stiefel-Whitney} de $\rho$.
\item Si $\rho$ est symplectique, virtuelle de degr\'e $0$ et de d\'eterminant $1$, on a
encore $W(\rho)^2=1$; le signe de $W(\rho)$ est intimement li\'e \`a la structure de module
galoisien de l'anneau des entiers de $L$, o\`u $\rho$ se factorise par $Gal(L/K)$ (Fr\"ohlich,
Taylor\dots voir par exemple \cite{tous}).
\item Dans \cite{dh}, Deligne et Henniart \'etudient comment $W(\rho)$ varie quand on tord
$\rho$ par une repr\'esentation  pas trop ramifi\'ee. Ils posent aussi une question relative
\`a la trivialit\'e de $W$ sur certains produits tensoriels de repr\'esentations (question
4.10), et y r\'epondent positivement dans certains cas particuliers.
\item Dans \cite{henniart}, Henniart \'etend la version de G\'erardin-Kutzko du th\'eor\`eme
de Lamprecht et Dwork mentionn\'e ci-dessus \`a certaines repr\'esentations galoisiennes
``sauvages homog\`enes". Pour une telle repr\'esentation $\rho$, il obtient une d\'ecomposition
du type
\begin{equation}\label{eq2}
W(\rho) \equiv \det_\rho(g_\rho)^{-1} G(g_\rho)^{\deg \rho} \pmod{\mu}
\end{equation}
o\`u $\mu$ est le groupe des racines $p$-primaires de l'unit\'e ($p=$ caract\'eristique
r\'esiduelle), $g_\rho$ est essentiellement un \'el\'ement de $K^*$ attach\'e \`a $\rho$ et
$G(g_\rho)$ est une somme de Gauss quadratique attach\'ee \`a $g_\rho$. Par convention
$G(g_\rho)=1$ si
$p=2$, et
$G(g_\rho)$ est une racine $4$-i\`eme de l'unit\'e si $p>2$.
\item Dans \cite{volf1} et \cite{volf2}, Volf r\'esoud affirmativement certains cas de 
\cite[qu. 4.10]{dh}.
\end{enumerate}

Dans cette note, je m'int\'eresse principalement \`a la partie \emph{$p$-primaire} de
$W(\rho)$, lorsque ce nombre est une racine de l'unit\'e, o\`u $p$ est la caract\'eristique
r\'esiduelle. L'observation centrale est tr\`es simple: en \'etudiant l'action de
$Gal(\Q^\ab/\Q)$ sur $W(\rho)$, on peut borner l'ordre de cette partie $p$-primaire en fonction
du corps de d\'efinition de $\rho$. Voir proposition \ref{p1} et ses corollaires. Les calculs indiquent en fait que le comportement galoisien est tr\`es diff\'erent sur trois facteurs (canoniques?) de $W(\rho)$ g\'en\'eralisant \eqref{eq2}.

Une premi\`ere version de ces r\'esultats a \'et\'e obtenue en 1986 \cite{ktoh}. Ils ont
ensuite \'et\'e raffin\'es au cours des ann\'ees 80/90, suivant notamment des suggestions de
Guy Henniart, mais je ne les ai jamais soumis
\`a publication. J'esp\`ere qu'il n'est pas trop tardif de le faire maintenant: je remercie
Claude L\'evesque pour son amicale insistance sur ce point. Ce texte est plus ou moins le
contenu des notes d'une s\'erie d'expos\'es que j'ai faits \`a l'Universit\'e McMaster en 1989
ou 1991, avec des ajouts pour $p=2$.

Il est tentant de se demander si ces calculs \'eclairent la question 4.10 de \cite{dh}. J'ai
jug\'e plus prudent de m'en abstenir ici, pour ne pas retarder encore plus la r\'edaction.

\subsection*{Notations} $K$ est un corps local non archim\'edien, de caract\'eristique $0$ et
de caract\'eristique r\'esiduelle $p$. Son corps r\'esiduel est $\F_q$. Selon Langlands et
Deligne, on associe \`a une repr\'esentation complexe $\rho$ du groupe de Weil de $K$ un
\emph{facteur epsilon} $\epsilon(\rho,\psi,dx,s)$, d\'ependant \'egalement d'un caract\`ere
additif $\psi$ de $K$, d'une mesure de Haar $dx$ sur $K$ et d'une variable complexe $s$. Comme
dans \cite{tate}, on choisit pour $\psi$ le caract\`ere additif standard:
\[\begin{CD}
\psi_K:K@>\Tr_{K/\Q_p}>>\Q_p@>>> \Q_p/\Z_p &\Inj& \Q/\Z @>\exp{2\pi i-}>> \C^*
\end{CD}\]
pour $dx$ la mesure de Haar autoduale pour $\psi$, et on pose
\[W(\rho)=\epsilon(\rho,\psi_K,dx,1/2).\]

Ces normalisations semblent importantes au moins pour certains calculs, notamment au \S \ref{s4}.

Je ne consid\'ererai que des repr\'esentations galoisiennes, c'est-\`a-dire se prolongeant au
groupe de Galois absolu de $K$: cela ne restreint pas vraiment la g\'en\'eralit\'e puisqu'une
repr\'esentation du groupe de Weil s'obtient par torsion \`a partir d'une repr\'esentation
galoisienne.

On note $\ff(\rho)$ le conducteur de $\rho$ et $c(\rho)$ son exposant. On note
$\Gamma=Gal(\Q^\ab/\Q)$, $\kappa:\Gamma\iso \hat \Z^*$ le caract\`ere cyclotomique et
$\kappa_p:\Gamma\to \Z_p^*$ sa composante $p$-primaire. On note $E=\Q(\rho)\subset \Q^\ab$ le
corps engendr\'e par les valeurs du caract\`ere de $\rho$ et $\Gamma_E=Gal(\Q^\ab/E)$. Si
$a,b\in K^*$, on note $(a,b)\in \{\pm 1\}$ leur symbole de Hilbert.

\section{Remarques sur une racine $4$-i\`eme de l'unit\'e}

Supposons $p>2$. On peut se demander s'il existe une fonction $\rho\mapsto \iota(\rho)$, \`a
valeurs dans les racines $4$-i\`emes de l'unit\'e, co\"\i ncidant avec $G(g_\rho)^{\deg \rho}$
dans le cas de \eqref{eq2} et v\'erifiant les conditions (i) et (ii) de l'introduction. Je
n'\'etudierai pas ce probl\`eme ici, mais me contenterai de deux remarques:

\begin{lemme}\label{l1} a) Dans \eqref{eq2}, on a
\[G(g_\rho)^{2\deg \rho} =  (-1, \ff(\rho))\]
o\`u $\ff(\rho)$ d\'esigne un g\'en\'erateur quelconque du conducteur de $\rho$. Cette
fonction v\'erifie les conditions {\rm (i)} et {\rm (ii)} de l'introduction.\\ 
b) Soit $K_0$ un
corps de nombres, et soit $\rho$ une repr\'esentation galoisienne complexe de $G_{K_0}$. Pour
toute place $v$ de $K_0$, notons $\rho_v$ la restriction de $\rho$ au groupe de d\'ecomposition
en $v$. Alors le produit
\[\prod_{v \text{ finie impaire}} (-1, \ff(\rho_v))_v\]
ne d\'epend pas que des composantes dyadiques et archim\'ediennes de $\rho$.
\end{lemme}

\begin{proof} a) r\'esulte facilement de la d\'efinition de $G(g_\rho)$ et de $g_\rho$
\cite[\S\S 2, 4]{henniart}. La seconde affirmation est \'evidente. 

Pour b),  il suffit de donner un exemple. On prend $K_0=\Q$; soit $l$ un
nombre premier impair, et soient $p_1,p_2$ deux nombres premiers v\'erifiant
\begin{thlist}
\item $p_1\equiv p_2\equiv 1\pmod{l}$;
\item $l\mid \displaystyle\frac{p_i-1}{d_i}$, o\`u $d_i$ est l'ordre de $2$ dans $\F_{p_i}^*$;
\item $p_1\equiv 1\pmod{4}$, $p_2\equiv -1\pmod{4}$.
\end{thlist}

Pour $i=1,2$, soit $\chi_i$ un caract\`ere de Dirichlet $\pmod{p_i}$ d'ordre $l$: leur
existence est assur\'ee par (i). Comme $l$ est impair, $\chi_1$ et $\chi_2$ sont triviaux \`a
l'infini, et (ii) implique qu'ils sont aussi triviaux en $2$. D'autre part, $(-1,
\ff(\chi_i))_p=1$ si $p\ne p_i$, et (iii) implique que $(-1, \ff(\chi_1))_{p_1}=1$, $(-1,
\ff(\chi_2))_{p_2}=-1$.

Un exemple minimal est $l=3$, $p_1=109$, $p_2=31$.
\end{proof}

Le sens du lemme \ref{l1} a) est que l'obstruction \`a \'etendre $\rho\mapsto G(g_\rho)^{\deg
\rho}$ dans le style de l'introduction dispara\^\i t quand on l'\'el\`eve au carr\'e. Le lemme
\ref{l1} b) sugg\`ere qu'on ne peut gu\`ere attendre une preuve d'existence de cette extension
par la m\'ethode de Deligne \cite{deligne}.

Pour la suite je me contenterai d'une version de $\iota(\rho)$ au signe pr\`es.
Une telle fonction est facile \`a exhiber: on peut prendre
\[\iota(\rho) = i^{\left(\frac{N(\ff(\rho))-1}{2}\right)^2}.\]

\begin{defn} Si $p>2$, 
\begin{equation}\label{eq3}
W^*(\rho)=\iota(\rho) W(\rho).
\end{equation}
Si $p=2$, $W^*(\rho)=W(\rho)$.
\end{defn}

L'inconv\'enient de ce choix de $\iota(\rho)$ est qu'il est invariant par l'action de Galois: ce n'est pas le
cas de l'invariant $G(g_\rho)^{\deg \rho}$ de Henniart. On peut esp\'erer qu'une r\'eponse
positive
\`a la question ci-dessus permettrait de se d\'ebarrasser compl\`etement des d\'esagr\'eables
symboles de Hilbert polluant les formules ci-dessous.

\section{Action galoisienne sur $W^*(\rho)$} 

Toute repr\'esentation galoisienne complexe
$\rho$ est d\'efinie sur un corps cyclotomique; le groupe de Galois $\Gamma=Gal(\Q^\ab/\Q)$
op\`ere donc sur [le caract\`ere de] $\rho$, ainsi que sur $W(\rho)\in \Q^\ab$.

Si $p$ est un nombre premier impair, posons $p^*=(-1)^{\frac{p-1}{2}} p$, de sorte que
$\Q(\sqrt{p^*})\subset \Q(\mu_p)$. Si $p=2$, posons $p^*=p$: on a $\Q(\sqrt{2})\subset
\Q(\mu_8)$.  Le
lemme suivant se v\'erifie facilement (pour $p=2$, utiliser la relation $(2,-1)=1$):

\begin{lemme} Pour tout nombre premier $p$ et tout $\sigma\in \Gamma$, on a
\[\sqrt{p^*}^{\sigma-1} = (p,\kappa_p(\sigma))\]
(symbole de Hilbert calcul\'e dans $\Q_p$).\qed
\end{lemme}

On en d\'eduit:

\begin{prop}\label{p1} Soit $\sigma\in \Gamma$. Alors
\[W^*(\rho)^\sigma = W^*(\rho^\sigma)\det_\rho(\kappa_p(\sigma))^\sigma
(N\ff(\rho),\kappa_p(\sigma))\] o\`u $\kappa_p(\sigma)\in \Z_p^*$ est le caract\`ere
$p$-cyclotomique de $\sigma$ et $\det_\rho$ est la repr\'esentation d\'eterminant de $\rho$,
vue comme caract\`ere multiplicatif via le corps de classes local.
\end{prop}

\begin{proof} C'est une reformulation du th\'eor\`eme de Fr\"ohlich \cite[p. 43, cor.
5.2]{martinet}. Avoir remplac\'e $W(\rho)$ par $W^*(\rho)$ donne une formule l\'eg\`erement
plus propre.
\end{proof}

\begin{cor}\label{c1} Soit $E\subset \Q^\ab$ le corps engendr\'e par les valeurs du caract\`ere
de
$\rho$. Notons $p^n$ le nombre exact de racines $p$-primaires de l'unit\'e contenues dans $E$.\\
a) Si $p>2$ et $n>0$,  $W^*(\rho)\in \mu_{p^{n+1}} E^*$.\\
b) Si $p>2$ et $n=0$, $W^*(\rho)\in E(\mu_p)^*$ et $W^*(\rho)^{2d}\in E^*$, o\`u $d =
\pgcd(|\mu_{p-1}\cap E|,[E(\mu_p):E])$.\\
c) Si $p=2$ et $n\ge 2$, $W^*(\rho)\in \mu_{2^{n+2}} E^*$.\\
d) Si $p=2$ et $n=1$,  $W^*(\rho)\in E(\mu_8)^*$ et $W^*(\rho)^2\in E^*$.
\end{cor}

\begin{proof} Soit $\sigma\in \Gamma_E:=Gal(\Q^\ab/E)$. Comme $\det_\rho$ prend ses valeurs
dans $E$, on a:
\[W^*(\rho)^{\sigma-1} = \det_\rho(\kappa_p(\sigma))(N\ff(\rho),\kappa_p(\sigma)).\]

a) L'hypoth\`ese implique que $\kappa_p(\sigma)\equiv 1\pmod{p^n}$. 

Si $p>2$, 
$\kappa_p(\sigma)$ est un carr\'e et le symbole de Hilbert vaut $1$. De plus, $\det_\rho(\kappa_p(\sigma))\in \mu_p$: en effet $\det_\rho(1+p\Z_p)\subset \mu_{p^n}$ et
$1+p^n\Z_p=(1+p\Z_p)^{p^{n-1}}$. On trouve donc que $W^*(\rho)^{\sigma-1}\in \mu_p$, et
$W^*(\rho)^{\sigma-1}=1$ si $\kappa_p(\sigma)\in 1+p^{n+1}\Z_p$. Il en r\'esulte que
$W^*(\rho)\in E(\mu_{p^{n+1}})$ et
$W^*(\rho)^p\in E$, donc que $W^*(\rho)\in
\mu_{p^{n+1}} E^*$.

b) On a $\det_\rho(u)=1$ et $(p,u)=1$ si $u\in 1+p\Z_p$: ceci montre
que
$W^*(\rho)\in E(\mu_p)$. Posons maintenant $d_1=|\mu_{p-1}\cap E|$,
$d_2=[E(\mu_p):E]$ et $d=\pgcd(d_1,d_2)$. Alors
$\det_\rho(\kappa_p(\sigma))\in \mu_d$ pour tout $\sigma\in \Gamma_E$, puisque $\sigma^{d_2}\in
\Gamma_{E(\mu_p)}$. De plus, on voit que $W^*(\rho)^d\in E$ si $d$ est pair et
$W^*(\rho)^{2d}\in E$ si
$d$ est impair.

c) Si $p=2$, le raisonnement de a) reste valable tant que $n\ge 3$ ($(\Z_2^*)^{2^{n-2}} = 1+2^n\Z_2$
$\Rightarrow$ $\det\rho(1+2^n\Z_2)\subset \mu_4$), et m\^eme 
pour $n=2$ (alors $\det\rho(u)\in \mu_4$ pour tout $u\in \Z_2^*$). 

d) M\^eme calcul que pr\'ec\'edemment.
\end{proof}

\begin{cor} \label{c2} Gardons les notations du corollaire \ref{c1}, et supposons que $\det_\rho$ soit trivial. Alors $W^*(\rho)\in E^*$, sauf si 
\begin{thlist}
\item $p>2$, $n=0$ et $N\ff(\rho)$ n'est pas un carr\'e; alors $W^*(\rho)\in \sqrt{p^*}E^*$.
\item $p=2$, $n=2$ et $N\ff(\rho)$ n'est pas un carr\'e; alors $W^*(\rho)\in \mu_8 E^*$.
\item $p=2$, $n=1$.
\end{thlist}
\end{cor}

\begin{proof} On reprend la d\'emonstration du corollaire \ref{c1}.
\end{proof}

\begin{rque} Supposons $\rho$ somme de repr\'esentations homog\`enes au sens de \cite[\S
4]{henniart}. Si
$p>2$, on a n\'ecessairement
$n>0$: en effet, la restriction d'une composante irr\'eductible au dernier saut de ramification
est une somme de caract\`eres de degr\'e $1$, tous \'egaux entre eux, sur un
$p$-groupe (\emph{loc.cit.}). Lorsque $p=2$ il existe des caract\`eres quadratiques sauvages,
par exemple pour
$K=\Q_2$ celui donn\'e par l'extension $K(\sqrt{2})/K$.
\end{rque}

\begin{cor}\label{c3}  Gardons les notations du corollaire \ref{c1}, et supposons que
$W^*(\rho)$ soit une racine de l'unit\'e. Soit $m$ le nombre de racines de l'unit\'e de
$E$. Alors
\[W^*(\rho)^m=
\begin{cases}
1 &\text{si $p\nmid m$}\\
\det_\rho(1+m) &\text{si $p$ est impair et $p\mid m$}\\
\det_\rho(1+m) &\text{si $p=2$ et $8\mid m$}\\
\det_\rho(1+m)(-1)^{v_2(N\ff(\rho))} &\text{si $p=2$ et $8\nmid m$.}
\end{cases}\]
\end{cor}

\begin{proof} Supposons d'abord $p>2$. D'apr\`es le corollaire \ref{c1} a) et b), on a
$W^*(\rho)\in \mu_{pm}$. 

Si $p\nmid m$, on a aussi $W^*(\rho)^{2d}\in \mu_m$, o\`u $2d$ est premier \`a $p$. Donc
$W^*(\rho)\in \mu_{m}$, d'o\`u l'\'enonc\'e.

Supposons $p\mid m$.  L'image de $\kappa(\Gamma_E)$ dans $(\Z/m)^*$ (\emph{resp.}  dans $(\Z/pm)^*$) est
triviale (\emph{resp.}  non triviale) puisque $\mu_m\subset E$ (\emph{resp.}  $\mu_{pm}\not\subset E$).
Comme $\Ker\left((\Z/pm)^*\to (\Z/m)^*\right)$ est cyclique d'ordre $p$ engendr\'e par $1+m$,
on peut choisir
$\sigma\in
\Gamma_E$ tel que
$W^*(\rho)^\sigma = W^*(\rho)^{1+m}$ et $\kappa_p(\sigma)=1+m$. En particulier
$\kappa_p(\sigma)\equiv 1\pmod{p}$, donc $\kappa_p(\sigma)$ est un carr\'e. Alors l'\'enonc\'e
d\'ecoule de la proposition
\ref{p1}.

Supposons maintenant $p=2$. On peut encore choisir $\sigma$ comme ci-dessus. D'apr\`es le
corollaire \ref{c1} a), c) et d), on a  $W^*(\rho)\in
\mu_{4m}$. Soit $n=v_2(m)$. Si $n\ge 3$, 
$1+m\equiv 1\pmod{8}$ est un carr\'e dans $\Q_2$, donc le m\^eme calcul que ci-dessus est
valable. Si $n=2$ (\emph{resp.}  $n=1$), $1+m\equiv
5\pmod{8}$ (\emph{resp.}  $1+m\equiv 3\pmod{8}$): dans les deux cas on a $(N\ff(\rho),1+m)_2=
(-1)^{v_2(N\ff(\rho))}$, d'o\`u la formule dans ces cas. 
\end{proof}

\section{Op\'erations d'Adams}

\subsection{Rappel} Soit $G$ un groupe fini, et notons $R(G)$ l'anneau des repr\'esentations
complexes de $G$. Si $\rho\in R(G)$ est de caract\`ere $\chi$ et si $k\in\Z$, on note $\Psi^k
\chi$ la fonction centrale d\'efinie par
\[\Psi^k\chi(g) = \chi(g^k).\]

On d\'emontre (par exemple \cite[\S 9.1, ex. 3]{serre}) que c'est le caract\`ere d'une
repr\'esentation, not\'ee $\Psi^k\rho$.

Supposons maintenant $k$ premier \`a l'ordre de $G$. Alors $\Psi^k $ ne d\'epend que de $k$
modulo $e$, o\`u $e$ est l'exposant de $G$. On obtient ainsi une action de $(\Z/e)^*$ sur
$R(G)$.

D'autre part, les caract\`eres de $G$ prennent leurs valeurs dans $F=\Q(\mu_e)$. On a donc
aussi une action de $Gal(F/\Q)$ sur $R(G)$, et

\begin{lemme}\label{l2} Soit $\kappa:Gal(F/\Q)\iso (\Z/e)^*$ le caract\`ere
cy\-clo\-to\-mi\-que. Pour tout $\rho\in R(G)$ et tout $\sigma\in Gal(F/\Q)$, on a $\rho^\sigma
=\Psi^{\kappa(\sigma)}\rho$.\qed
\end{lemme}

Ceci montre que $\Psi^k\rho$ ne d\'epend que de l'image de $k$ dans
$\hat\Z^*/\kappa(\Gamma_E)$, o\`u $E=\Q(\rho)$.

\subsection{Action des op\'erations d'Adams sur les constantes locales} En combinant le lemme
\ref{l2} et la proposition \ref{p1}, on obtient:

\begin{prop}\label{p2} Soit $\rho$ une repr\'esentation galoisienne complexe de $K$. Soit $L/K$ une extension finie galoisienne par laquelle se factorise $\rho$, et soit $k$ un entier premier \`a $d=[L:K]$. Notons $d_p$ la plus grande puissance de $p$ divisant $d$. Alors
\[W^*(\Psi^k \rho)= W^*(\rho)^\sigma \det_\rho(k_p)^{-k}
(N\ff(\rho),k_p)\] 
o\`u $\sigma\in Gal(\Q(\mu_d)/\Q)$ est un \'el\'ement de caract\`ere
cyclotomique congru \`a $k$ modulo $d$, et $k_p$ est la projection de $k$ dans $(\Z/d_p)^*$.\qed
\end{prop}

On en d\'eduit le corollaire suivant, qui pr\'ecise le corollaire \ref{c3}:

\begin{cor}\label{c4} Gardons les notations de la proposition \ref{p2} et supposons que
$W^*(\rho)$ soit une racine de l'unit\'e.  Si $k$ est premier \`a $pd$ (en particulier si $p$
divise $d$), on a
\[W^*(\Psi^k \rho)= W^*(\rho)^k \det_\rho(k_p)^{-k}
(N\ff(\rho),k_p).\]
En particulier, si $k\in \kappa(\Gamma_E)$, on a
\[W^*(\rho)^{k-1}= \det_\rho(k_p)
(N\ff(\rho),k_p).\]
\end{cor}

Voici un exemple amusant:

\begin{cor} \label{c5} Dans le corollaire \ref{c4}, supposons que $\mu_p\subset E$ mais
$\mu_{p^2}\not\subset E$ (par exemple que $\rho$ soit un caract\`ere d'ordre
$p$). Alors
\[W^*(\Psi^k \rho)= W^*(\rho)^{k^p}(N\ff(\rho),k_p).\]
\end{cor}

\begin{proof} L'hypoth\`ese signifie que $\kappa_p(\Gamma_E)=1+p\Z_p$. En particulier
$k_p^{p-1}\in \kappa_p(\Gamma_E)$ et d'apr\`es la seconde formule du corollaire
\ref{c4}:
\[\det_\rho(k_p)=\det_\rho(k_p)^{1-p}=\det_\rho(k_p^{p-1})^{-1} =
\left(W^*(\rho)^{k^{p-1}-1}\right)^{-1}.\]

En reportant dans la premi\`ere formule du corollaire \ref{c4}, on trouve
\begin{multline*}
W^*(\Psi^k \rho)= W^*(\rho)^k \left(W^*(\rho)^{k^{p-1}-1}\right)^k
(N\ff(\rho),k_p)\\
=W^*(\rho)^{k^p}
(N\ff(\rho),k_p).
\end{multline*}
\end{proof}

\begin{rque} Si $\rho$ est un caract\`ere d'ordre $p$, la formule  du corollaire \ref{c5}
s'\'etend aux valeurs de $k$ divisibles par $p$ gr\^ace au corollaire \ref{c3} (en posant
$k_p=1$). Ce genre de formule ne semble pas se g\'en\'eraliser \`a des caract\`eres d'ordre
$p^2$ (voir proposition \ref{p3}).

Dans le num\'ero suivant on dira quelque chose sur l'action d'une op\'eration d'Adams $\Psi^k$
lorsque $k$ divise l'ordre d'un groupe de d\'efinition de $\rho$ dans le cas o\`u $p>2$ et $K$
est absolument non ramifi\'e, \emph{cf.} corollaire \ref{c6}.
\end{rque}

\section{Exemple: caract\`eres logarithmiques}\label{s4}

Dans \cite{GK}, une exponentielle tronqu\'ee pointe le bout du nez (pour $p=2$); elle appara\^\i t
explicitement dans \cite{dh} pour des raisons diff\'erentes. Nous allons utiliser ici un ``vrai" logarithme.

Soient $U$ le groupe des unit\'es de $K$ et $U_1$ le groupe des unit\'es principales. Rappelons
que la fonction logarithme converge sur $U_1$ et d\'efinit un homomorphisme continu
\[\log:U_1\to K\]
de noyau les racines $p$-primaires de l'unit\'e. Il est habituel (Iwasawa) de prolonger $\log$
en un homomorphisme sur
$K^*$ tout entier comme suit:

\begin{itemize}
\item $\log\zeta=0$ si $\zeta$ est une racine de l'unit\'e (cette condition est
n\'e\-ces\-sai\-re);
\item Soit $x\in K^*$. Si $e$ est l'indice de ramification absolu de $K$, on a 
$x^e=p^nu$ avec
$u\in U$ et $n\in \Z$. Alors
$\log x = \frac{1}{e}\log u$.
\end{itemize}

\begin{defn} Soit $\alpha\in K$. Pour $x\in K^*$, on note
\[\chi_\alpha(x)=\psi_K(\alpha\log x).\]
C'est le \emph{caract\`ere logarithmique} attach\'e \`a $\alpha$.
\end{defn}

Tout caract\`ere logarithmique est trivial sur les racines de l'unit\'e et les puissances
fractionnaires de $p$, et prend des valeurs
$p$-primaires. R\'eciproquement:

\begin{lemme} L'homomorphisme
\begin{align*}
K&\to \Hom(K^*/(\mu(K)+\langle p\rangle^\Q\cap K^*),\mu_{p^\infty}(\C))\\
\alpha&\mapsto \chi_\alpha
\end{align*}
est surjectif.
\end{lemme}

\begin{proof} C'est clair par dualit\'e additive, puisque l'image de $\log$ est un sous-groupe
ouvert de $K$.
\end{proof}

Malheureusement le noyau de $\alpha\mapsto \chi_\alpha$ d\'epend fortement de la ramification
absolue de $K$. De m\^eme il est difficile d'\'evaluer le conducteur de $\chi_\alpha$ en
g\'en\'eral. Pour ces raisons, je me limite maintenant au cas o\`u
$K/\Q_p$ est
\emph{non ramifi\'e}.\footnote{Cette restriction n'est peut-\^etre pas trop d\'eraisonnable:
\'etant donn\'e la propri\'et\'e  d'inductivit\'e de $W$, il suffit en fait d'\'etudier ces
constantes dans le cas particulier $K=\Q_p$. C'est aussi raisonnable d'un point de vue
motivique.} On a alors un r\'esultat assez agr\'eable (en tout cas pour $p>2$):

\begin{prop}\label{p3} Supposons $K$ non ramifi\'e sur $\Q_p$. \\
a) Pour $\alpha\in K$ on a 
$\chi_\alpha=1$ $\iff$ $|2\alpha|\le p$. Si $|2\alpha|> p$, $\chi_\alpha$ est de conducteur
$(\alpha^{-1})$ et d'ordre
$|2\alpha|/p$.\\ 
b) Si $p>2$, on a:
\[W(\chi_\alpha)=G(\alpha) \psi_K(\alpha(1-\log \alpha))\]
o\`u $G(\alpha)$ est la somme de Gauss quadratique normalis\'ee
de \cite[p. 352]{GK}  (racine $4$-i\`eme de l'unit\'e).\\
b) Si $p=2$ et $v(\alpha)$ est impair, on a
\[W(\chi_\alpha)=G(\chi_\alpha) \psi_K(\alpha(1-\log \alpha))\]
o\`u $G(\chi_\alpha)$ est la somme de Gauss quadratique normalis\'ee
de \cite[p. 352]{GK}  (racine $8$-i\`eme de l'unit\'e). Si $v(\alpha)$ est pair et $\le -6$, on
a
\[W(\chi_\alpha)= \psi_K(\alpha(1-\log \alpha)-2^{n-3}\alpha^{2F^{-1}-1})\]
o\`u $F$ est l'automorphisme de Frobenius absolu de $K$. (Voir la d\'e\-mons\-tra\-tion pour une
formule dans le cas $v(\alpha)=-4$.)
\end{prop}

\begin{proof} a) d\'ecoule du fait que $\log K^*=2pO_K$ et que
$v(\log(1+x))=v(x)$ si
$v(x)>
\frac{e}{p-1}$ o\`u $e$ est l'indice de ramification absolu (ici, $e=1$). Pour b) et c),
on utilise la formule de \cite[1.4]{GK}
\[W(\chi_\alpha)=G(\chi_\alpha)\chi_\alpha(d)\psi_K(d^{-1})\]
pour un $d\in K$ tel que
\[\chi_\alpha(1+x) = \psi_K(d^{-1}x)\]
pour tout $x$ tel que $v(x)\ge c - [c/2]$, o\`u $c=v(\ff(\chi_\alpha))$ et $G(\chi_\alpha)$ est
une racine $4$-i\`eme ou $8$-i\`eme de l'unit\'e, valant $1$ pour $c$ pair (\emph{cf.}
\cite[prop. 1]{tate}). Si $p>2$, $G(\chi_\alpha)$ ne d\'epend que de $d$ \cite[p. 352]{GK}.

Pour $p>2$, on voit tout de suite que
$d^{-1}=\alpha$ convient, et la formule r\'esulte alors de la d\'efinition de $\chi_\alpha$.
Pour $p=2$, on s'int\'eresse \`a $\psi_K(\alpha\log(1+x))$. \'Ecrivons $\alpha=a/2^n$ avec
$a\in U$ et $n\ge 3$ (voir a)), d'o\`u $c=n$. Pour $v(x)\ge n-[n/2]$ on a
\[\log(1+x)\equiv x-x^2/2\pmod{2^nO_K}\]
et m\^eme $\log(1+x)\equiv x\pmod{2^nO_K}$ si $n$ est impair. Dans ce cas, on peut choisir
$d^{-1}=\alpha$ comme en b). Si $n$ est pair, on \'ecrit $x=2^{n/2}u$; alors $x^2=2^nu^2$, donc
\begin{multline*}
\psi_K(\alpha\frac{x^2}{2})=\psi_K(\frac{a}{2^n}\frac{x^2}{2}) =
\psi_K(\frac{a}{2}u^2) =\psi_K(\frac{a}{2}u^F)\\ = \psi_K(\frac{a^{F^{-1}}}{2}
u)
=\psi_K(\frac{a^{F^{-1}}}{2^{n/2+1}}x)=\psi_K(2^{n/2-1}\alpha^{F^{-1}}x)
\end{multline*}
donc on peut choisir $d^{-1} = \alpha -2^{n/2-1}\alpha^{F^{-1}}$, d'o\`u
\[\chi_\alpha(d)\psi_K(d^{-1})=\psi_K(-\alpha\log(\alpha -2^{n/2-1}\alpha^{F^{-1}})+\alpha
-2^{n/2-1}\alpha^{F^{-1}}).\]

Supposons maintenant $n\ge 6$. Alors on peut \'ecrire
\begin{multline*}
\log(\alpha -2^{n/2-1}\alpha^{F^{-1}})=\log\alpha
+\log(1-2^{n/2-1}\alpha^{F^{-1}-1})\\
\equiv\log\alpha
-2^{n/2-1}\alpha^{F^{-1}-1}-2^{n-3}\alpha^{2(F^{-1}-1)}\pmod{2^nO_K}
\end{multline*}
d'o\`u
\[
\chi_\alpha(d)\psi_K(d^{-1})
=\psi_K(-\alpha(1-\log\alpha)
-2^{n-3}\alpha^{2F^{-1}-1})
\]
comme souhait\'e.
\end{proof}

Supposons $p>2$. Comme $K$ ne contient pas de racines $p$-i\`emes de l'unit\'e, tout caract\`ere
sauvage $\chi$ de $K^*$ s'\'ecrit de mani\`ere unique comme produit d'un caract\`ere
mod\'er\'ement ramifi\'e $\chi_0$ et d'un caract\`ere
$\chi_\alpha$. Ceci permet d'\'ecrire une formule explicite pour $W(\chi)$ \`a partir de la
proposition \ref{p3}: si $\chi=\chi_0\chi_\alpha$, on trouve
\[W(\chi) = \chi_0(\alpha)W(\chi_\alpha) = \chi_0(\alpha)G(\alpha)\psi_K(\alpha(1-\log
\alpha))\]
cf. \cite[p. 98 cor. 2]{tate}. 

Ainsi $W(\chi)$ se d\'ecompose canoniquement en produit de trois facteurs: une racine de
l'unit\'e $\chi_0(\alpha)$, une racine $4$-i\`eme de l'unit\'e
$G(\alpha)$ et une racine $p$-primaire de l'unit\'e $\psi_K(\alpha(1-\log
\alpha))$. Notons cette derni\`ere $W_p(\chi)$.

\begin{cor}\label{c6} Si $p>2$, on a
\[W_p(\chi^p)=W_p(\chi)^p\]
tant que $\chi^p$ est sauvage (c'est-\`a-dire que $\chi_\alpha^p\ne 1$).\qed
\end{cor}

\begin{cor} \label{c7} Soit $p>2$.\\
a) Supposons $\alpha=a/p^2$, o\`u $a\in O_K$. Alors
\[W_p(\chi_\alpha) = \psi_K(\frac{a^p}{p^2}).\]
b) Pour $a_1,\dots,a_p\in O_K$, on a
\[W_p((1-\chi_{a_1/p^2})\dots(1-\chi_{a_p/p^2})) = \psi_K(\frac{a_1\dots a_p}{p}).\]
c) Pour $a_1,\dots,a_{p+1}\in O_K$, on a $W_p((1-\chi_{a_1/p^2})\dots(1-\chi_{a_{p+1}/p^2}))
=1$.\\
d) Supposons $K=\Q_p$. Pour tout $n\ge 2$, on a $W_p((1-\chi_{1/p^n})^{p^{n-1}})= \exp(2\pi
i/p)$.
\end{cor}

\begin{proof} a) La formule est vraie si $a$ est divisible par $p$, puisqu'alors
$\chi_\alpha=1$. Sinon, \'ecrivons
$a=a_0(1+pu)$, avec
$a_0^{q-1}=1$ et
$u\in O_K$. Alors
$\log\alpha = pu+p^2v$ avec $v\in O_K$, et
\[W_p(\chi_\alpha) = \psi_K(\frac{a_0}{p^2}(1+pu)(1-pu+p^2v))=\psi_K(\frac{a_0}{p^2}).\]

D'autre part $a_0= a_0^q\equiv a^q\pmod{p^2}$, donc $\frac{a_0}{p^2}\equiv
\frac{a^q}{p^2}\pmod{O_K}$. Mais soit $F$ l'automorphisme de Frobenius de $K$. On a
\begin{align*}
a^p&\equiv a^F\pmod{p}\\
\intertext{d'o\`u}
a^{p^{n+1}}&\equiv a^{p^nF} \pmod{p^{n+1}} \quad \forall n\ge 0\\
\intertext{donc}
\Tr_{K/\Q_p}(a^{p^{n+1}})&\equiv \Tr_{K/\Q_p}(a^{p^n}) \pmod{p^{n+1}} \quad \forall n\ge 0
\intertext{et par r\'ecurrence}
\Tr_{K/\Q_p}(a^q)&\equiv \Tr_{K/\Q_p}(a^p) \pmod{p^2}
\intertext{d'o\`u finalement}
\psi_K(\frac{a_0}{p^2})&=\psi_K(\frac{a^p}{p^2}).
\end{align*}

b) r\'esulte imm\'ediatement de a) puisque la forme polaire de $a^p$ est $p!a_1\dots a_p$ et
que $p!\equiv -1\pmod{p}$ (th\'eor\`eme de Wilson). c) r\'esulte aussi de a) (ou de b)). 

Pour d), posons $\chi=\chi_{1/p^n}$. D'apr\`es la proposition \ref{p3}, on a
\[W_p(\chi)=\psi_{\Q_p}(1/p^n)=\exp(2\pi i/p^n).\]

 On \'ecrit
\[(1-\chi)^{p^{n-1}} = \sum_{i=0}^{p^{n-1}} (-1)^i \binom{p^{n-1}}{i} \chi^{i}.\]

Pour $i=0$ et $i=p^{n-1}$ on a $\chi^i=1$, sinon $\chi^i\ne 1$; de plus $\chi^{p^{n-2}}$ est
d'ordre $p$. Soient
$i\in ]0,p^{n-1}[$ et
$t=v_p(i)$. Alors $v_p(\binom{p^{n-1}}{i}) =n-1-t$. Posons $i=p^ti_0$ et
$\binom{p^{n-1}}{i}=p^{n-1-t}u$. Alors, en utilisant les corollaires \ref{c1} a), \ref{c5} et
\ref{c6}:
\begin{multline*}
W_p(\binom{p^{n-1}}{i}\chi^i) = W_p(\chi^i)^{\binom{p^{n-1}}{i}} =
W_p(\chi^{p^ti_0})^{p^{n+1-t}u}=W_p(\chi^{p^{n-2}i_0})^{pu}\\
=W_p(\chi^{p^{n-2}})^{i_0^ppu}=W_p(\chi^{p^{n-2}})^{i_0pu}=W_p(\chi)^{p^ti_0p^{n-1-t}u}
=W_p(\chi)^{i\binom{p^{n-1}}{i}}.
\end{multline*}

Par cons\'equent, $W_p((1-\chi)^{p^{n-1}}) = W_p(\chi)^A$, avec
\[A=\sum_{i=1}^{p^{n-1}-1} (-1)^i i\binom{p^{n-1}}{i}.\]

La somme $\sum_{i=1}^{p^{n-1}} (-1)^i i\binom{p^{n-1}}{i}$ est nulle, comme on le voit en
prenant la d\'eriv\'ee de $(1-X)^{p^{n-1}}$. Donc $A=p^{n-1}$ et
\[W_p((1-\chi)^{p^{n-1}}) = W_p(\chi)^{p^{n-1}}=\exp(2\pi i/p). \]
\end{proof}

\begin{rque} a) Le corollaire \ref{c6} devient faux pour $p=2$, m\^eme si $\chi$ est de la forme
$\chi_\alpha$ (voir proposition \ref{p3} c)). Dans le corollaire \ref{c7} a), le premier cas
pour $p=2$ serait $n=3$. Les caract\`eres $\chi_\alpha$ sont alors quadratiques, donc
$\alpha\mapsto W(\chi_\alpha)$ est une application quadratique d'apr\`es \cite[p. 126, cor.
2]{tate}. J'ai donn\'e une formule pour certains de ces caract\`eres quadratiques dans
\cite[th. 2]{ultra}, \`a savoir
\[W(\rho_u) = i^{\Tr_{K/\Q_2}(\frac{u-1}{2}^2)}\]
pour $u\in 1+2O_K$, o\`u $\rho_u(x):= (u,x)$ pour $x\in K^*$.
Je renonce \`a la comparer \`a celle de la proposition
\ref{p3} c), \`a commencer par d\'eterminer $u$ en fonction de $\alpha$\dots

b) Le corollaire
\ref{c7} a) donne aussi la formule suivante: 
\[W_p((1-\chi_{a_1/p^2})(1-\chi_{a_2/p^2})) = \psi_K(p^{-1}\frac{(a_1+a_2)^p-a_1^p-a_2^p}{p})\]
o\`u on reconna\^\i t la seconde composante du vecteur de Witt associ\'e \`a $a_1+a_2\pmod{p}$.

Pour $n>2$, je ne sais pas si la fonction $\chi\mapsto W_p(\chi)$ est polynomiale d'ordre
$p^{n-1}$ sur les caract\`eres d'ordre $\le p^{n-1}$: c'est sugg\'er\'e par le corollaire
\ref{c7} d).  Ceci est \`a comparer avec le
th\'eor\`eme 4.15 de
\cite{dh}.
\end{rque}

\end{document}